\documentclass[a4paper,12pt]{article}
\usepackage{a4wide}
\usepackage{amsmath}
\usepackage{amssymb}
\usepackage{amsthm}
\usepackage{latexsym}
\usepackage{graphicx}
\usepackage[english]{babel}
\usepackage{makeidx}
\usepackage[mathlines]{lineno}
\usepackage{mathtools}

\newtheorem{exm} [subsection]{Example}

\newtheorem{prop}[subsection]{Proposition}

\newtheorem{teor}[subsection]{Theorem}
\newtheorem{lema}[subsection]{Lemma}
\newtheorem{cor} [subsection]{Corollary}

\newcommand{\pa}{p_{\mathbf a}}
\newcommand{\paa}{p_{\mathbf a}}
\newcommand{\Pa}{P_{\mathbf a}}
\newcommand{\stir}{\genfrac{[}{]}{0pt}{}}

\newcommand{\DD}{\mathcal D}
\newcommand{\PP}{\mathcal P}

\def\lcm{\operatorname{lcm}}

\begin{document}
\selectlanguage{english}
\frenchspacing

\numberwithin{equation}{section}

\title{A note on plane partition diamonds}
\author{Mircea Cimpoea\c s$^1$ and Alexandra Teodor$^2$}
\date{}

\maketitle

\footnotetext[1]{ \emph{Mircea Cimpoea\c s}, University Politehnica of Bucharest, Faculty of
Applied Sciences, 
Bucharest, 060042, Romania and Simion Stoilow Institute of Mathematics, Research unit 5, P.O.Box 1-764,
Bucharest 014700, Romania, E-mail: mircea.cimpoeas@upb.ro,\;mircea.cimpoeas@imar.ro}
\footnotetext[2]{ \emph{Alexandra Teodor}, University Politehnica of Bucharest, Faculty of
Applied Sciences, 
Bucharest, 060042, E-mail: alexandra.teodor@upb.ro}

\begin{abstract}
We prove new formulas for $\DD_k(n)$, the number of plane partition diamonds of length $k$ of $n$,
and, also, for its polynomial part.


\end{abstract}

\maketitle

\section{Introduction}

In his famous book "Combinatory Analysis" \cite[Vol.II, Sect. VIII, pp. 91-170]{mahon} MacMahon introduced 
Partition Analysis as a computational
method for solving combinatorial problems in connection with systems of
linear Diophantine inequalities and equations. He considered
partitions of the form $(a_1,a_2,a_3,a_4)$ with 
\begin{equation}\label{e1}
a_1\geq a_2,\;a_1\geq a_3,\;a_2\geq a_4\text{ and }a_3\geq a_4.
\end{equation}
By using Partition Analysis he derived that
\begin{equation}\label{e2}
\sum x_1^{a_1}x_2^{a_2}x_3^{a_3}x_4^{a_4} = \frac{1-x_1^2x_2x_3}{(1-x_1)(1-x_1x_2)(1-x_1x_2x_3)(1-x_1x_2x_3x_4)},
\end{equation}
where the sum is taken over all non-negative integers $a_i$ satisfying \eqref{e1}.
Let $$\DD_1(n):=\#\{(a_1,a_2,a_3,a_4)\;:\;n=a_1+a_2+a_3+a_4\text{ where }a_i\text{ satisfy }\eqref{e1}\}.$$
Taking $x_1=x_2=x_3=x_4=q$ in \eqref{e2}, MacMahon observed that
\begin{equation}\label{e3} 
\sum_{n=0}^{\infty} \DD_1(n)q^n = \frac{1}{(1-q)(1-q^2)^2(1-q^3)}.
\end{equation}
In \cite{apr}, Andrews, Paule, and Riese introduce the family of plane partition diamonds, as a generalization
of the above example. A \emph{plane partition diamond} of length $k$ is a sequence of length $3k+1$ of nonnegative integers
$\mathbf a = (a_1,a_2,\ldots,a_{3k+1})$ satisfying, for $0 \leq i \leq k - 1$,
\begin{equation}\label{e4}
a_{3i+1}\geq a_{3i+2},\; a_{3i+1}\geq a_{3i+3},\; a_{3i+2}\geq a_{3i+4},\; a_{3i+3}\geq a_{3i+4}.
\end{equation}
Let $\DD_k(n)$ be the number of plane partitions diamonds of length $k$ of $n$. 
We mention that several generalizations of plane partition diamonds were studied in \cite{savage} and \cite{ap2} 
but are beyond the scope of this note.

The paper is organized as follows. In Section 2, we recall the definition and some basic
properties of the restricted partition function $\pa(n)$, where $\mathbf a = (a_1,\ldots, a_r)$ is a sequence
of positive integers. Also, we recall several results which would be used later on.

In Section 3, we study basic properties of the function $\DD_k(n)$.
For $k\leq 1$ we consider the sequence $\mathbf a[k]=(a[k]_1,a[k]_2,\ldots,a[k]_{3k+1})$, where 
$$
a[k]_j=\begin{cases} j,& j\not\equiv 4(\bmod\;6) \\ \frac{j}{2}, & j\equiv 4(\bmod\;6) \end{cases}.
$$

In Proposition \ref{p1} we show that $\DD_k(n)$ can be written as
$$\DD_k(n)= \sum\limits_{J\subset \{\alpha_k+1,\alpha_{k}+2,\cdots, k\}} p_{\mathbf a[k]}(n-m_J),$$
where $m_J=\sum_{i\in J}(3i-1)$ and $\alpha_k=\lfloor \frac{k+1}{2} \rfloor$.
In Proposition \ref{p2} we show that 
$$\DD_k(n)=f_{k,3k}(n)n^{3k}+\cdots+f_{k,1}(n)n+f_{k,0}(n)\text{ for }n\geq n_0(k),$$
where $n_0(k)$ is a constant which depends on $k$, is a quasi-polynomial of degree $3k$.
In Corollary \ref{t1} we obtain new formulas for the periodic functions $f_{k,j}$'s and,
consequently, for $\DD_k(n)$.

In Theorem \ref{main} we prove a concise formula of $\DD_k(n)$.
In Theorem \ref{tunde} we prove formulas for the 'Sylvester waves' associated to $\DD_k(n)$.
 Also, in Theorem \ref{t2} we prove a
concise formula of $\PP_k(n)$, the polynomial part of $\DD_k(n)$.


\section{Restricted partition function}

Let $\mathbf a := (a_1, a_2, \ldots , a_r)$ be a sequence of positive integers, $r \geq 1$. The \emph{restricted partition
function} associated to $\mathbf a$ is $\paa : \mathbb N \to \mathbb N$, $\paa(n) :=$ the number of integer solutions $(x_1, \ldots, x_r)$
of $\sum_{i=1}^r a_ix_i = n$ with $x_i \geq 0$. Note that the generating function of $\paa(n)$ is
\begin{equation}\label{gen}
\sum_{n=0}^{\infty}\paa(n)q^n= \frac{1}{(1-q^{a_1})\cdots(1-q^{a_r})}.
\end{equation}
Let $D$ be a common multiple of $a_1$, $a_2,\ldots,a_r$. We recall the following well known result:

\begin{prop}\label{quasi}(Bell \cite{bell})

$\paa(n)$ is a quasi-polynomial of degree $r-1$, with the period $D$, i.e.
$$\paa(n)=d_{\mathbf a,k-1}(n)n^{k-1}+\cdots+d_{\mathbf a,1}(n)n+d_{\mathbf a,0}(n),$$
where $d_{\mathbf a,m}(n+D)=d_{\mathbf a,m}(n)$ for $0\leq m\leq k-1$ and $n\geq 0$, and $d_{\mathbf a,k-1}(n)$ is
not identically zero.
\end{prop}

Sylvester \cite{sylvester,sylvesterc,sylv} decomposed the restricted partition in a sum of ``waves'': 
\begin{equation}\label{wave}
\paa(n)=\sum_{j\geq 1} W_{j}(n,\mathbf a), 
\end{equation}
where the sum is taken over all distinct divisors $j$ of the components of $\mathbf a$ and showed that for each such $j$, 
$W_j(n,\mathbf a)$ is the coefficient of $t^{-1}$ in
$$ \sum_{0 \leq \nu <j,\; \gcd(\nu,j)=1 } \frac{\rho_j^{-\nu n} e^{nt}}{(1-\rho_j^{\nu a_1}e^{-a_1t})\cdots (1-\rho_j^{\nu a_k}e^{-a_kt}) },$$
where $\rho_j=e^{\frac{2\pi i}{j}}$ and $\gcd(0,0)=1$ by convention. Note that $W_{j}(n,\mathbf a)$'s are quasi-polynomials of period $j$.
Also, $W_1(n,\mathbf a)$ is called the \emph{polynomial part} of $\paa(n)$ and it is denoted by $\Pa(n)$;
see also \cite[Section 4.4]{stanley}.

The \emph{unsigned Stirling numbers} are defined by
\begin{equation}
 \binom{n+r-1}{r-1}=\frac{1}{n(r-1)!}n^{(r)}=\frac{1}{(r-1)!}\left(\stir{r}{r}n^{r-1} + \cdots \stir{r}{2}n + \stir{r}{1}\right).
\end{equation}

We recall several results which would be used later on:

\begin{teor}(\cite[Theorem 2.8]{lucrare} and \cite{cori})\label{dam}
\begin{enumerate}
\item[(1)] For $0\leq m\leq r-1$ and $n\geq 0$ we have 
\begin{align*}
& d_{\mathbf a, m}(n) =\frac{1}{(r-1)!}  \sum_{\substack{0\leq j_1\leq \frac{D}{a_1}-1,\ldots, 0\leq j_r\leq \frac{D}{a_r}-1 \\ a_1j_1+\cdots+a_rj_r \equiv n (\bmod D)}} 
\sum_{k=m}^{r-1} \stir{r}{k+1} (-1)^{k-m} \binom{k}{m} \times \\
& \times D^{-k} (a_1j_1 + \cdots + a_rj_r)^{k-m}.
\end{align*}
\item[(2)] We have
\begin{align*}
& \pa(n)=\frac{1}{(r-1)!} \sum_{m=0}^{r-1}  \sum_{\substack{0\leq j_1\leq \frac{D}{a_1}-1,\ldots, 0\leq j_r\leq \frac{D}{a_r}-1 \\ a_1j_1+\cdots+a_rj_r \equiv n (\bmod D)}} \sum_{k=m}^{r-1} \stir{r}{k+1} (-1)^{k-m} \binom{k}{m} \times  \\
& \times D^{-k} (a_1j_1 + \cdots + a_rj_r)^{k-m}n^m.
\end{align*}
\end{enumerate}
\end{teor}

\begin{teor}(\cite[Corollary 2.10]{lucrare})\label{pan} 
We have
$$ \paa(n) = \frac{1}{(r-1)!} \sum_{\substack{0\leq j_1\leq \frac{D}{a_1}-1,\ldots, 0\leq j_r\leq \frac{D}{a_r}-1 \\ 
a_1j_1+\cdots+a_rj_r \equiv n (\bmod D)}} \prod_{\ell=1}^{r-1} \left(\frac{n-a_1j_{1}- \cdots -a_rj_r}{D}+\ell \right).$$
\end{teor}

\begin{prop}(\cite[Proposition 4.2]{remarks})\label{unde}
For any positive integer $j$ with $j|a_i$ for some $1\leq i\leq r$, we have that
$$ W_{j}(n,\mathbf a) = \frac{1}{D(r-1)!} \sum_{m=1}^r \sum_{\ell=1}^{j} \rho_j^{\ell} \sum_{k=m-1}^{r-1} 
\stir{r}{k+1} (-1)^{k-m+1} \binom{k}{m-1} \cdot$$ $$\cdot \sum_{\substack{0\leq j_1\leq \frac{D}{a_1}-1,\ldots, 0\leq j_r\leq \frac{D}{a_r}-1 \\ a_1j_1+\cdots+a_rj_r \equiv \ell (\bmod j)}} D^{-k} (a_1j_1+\cdots+a_rj_r)^{k-m+1} n^{m-1}.$$
\end{prop}

\begin{teor}(\cite[Corollary 3.6]{lucrare})\label{Pan}

For the polynomial part $\Pa(n)$ of the quasi-polynomial $\paa(n)$ we have
$$ \Pa(n) = \frac{1}{D(r-1)!} \sum_{0\leq j_1\leq \frac{D}{a_1}-1,\ldots, 0\leq j_r\leq \frac{D}{a_r}-1} 
\prod_{\ell=1}^{r-1} \left(\frac{n-a_1j_{1}- \cdots -a_rj_r}{D}+\ell \right).$$
\end{teor}

The \emph{Bernoulli numbers} $B_{\ell}$'s are defined by the identity
$$\frac{t}{e^t-1}=\sum_{\ell=0}^{\infty}\frac{t^{\ell}}{\ell !}B_{\ell}.$$
$B_0=1$, $B_1 = -\frac{1}{2}$, $B_2=\frac{1}{6}$, $B_4=-\frac{1}{30}$ and $B_n=0$ is $n$ is odd and $n\geq 1$.

\begin{teor}(\cite[Corollary 3.11]{lucrare} or \cite[page 2]{beck})\label{Pan2}

The polynomial part of $\paa(n)$ is
$$P_{\mathbf a}(n) := \frac{1}{a_1\cdots a_r}\sum_{u=0}^{r-1}\frac{(-1)^u}{(r-1-u)!}\sum_{i_1+\cdots+i_r=u} 
\frac{B_{i_1}\cdots B_{i_r}}{i_1!\cdots i_r!}a_1^{i_1}\cdots a_r^{i_r} n^{r-1-u}.$$
\end{teor}

\section{Preliminaries}

The number of plane partitions diamonds of length $k$ of $n$ is
$$\DD_k(n):=\#\{(a_1,a_2,\ldots,a_{3k+1})\;:\;n=a_1+a_2+\cdots+a_{3k+1}\text{ where }a_i\text{ satisfy }\eqref{e4}\}.$$
Using partition analysis, the authors in \cite{apr} find the generalization of \eqref{e3}, namely
\begin{equation}\label{e5}
\sum_{n=0}^{\infty} \DD_k(n)q^n = \frac{\prod\limits_{i=1}^k (1+q^{3i-1})}{\prod\limits_{i=1}^{3k+1}(1-q^i)}.
\end{equation}
Note that, if $i\leq \frac{k+1}{2}$ then $6i-2\leq 3k+1$. Since $(1+q^{3i-1})(1-q^{3i-1})=1-q^{6i-2}$, from \eqref{e5}
it follows that
\begin{equation}\label{e6}
\sum_{n=0}^{\infty} \DD_k(n)q^n = \frac{\prod\limits_{i=\alpha_k+1}^k (1+q^{3i-1})}
{\prod\limits_{i=1}^{\lfloor \frac{k+1}{2} \rfloor}(1-q^{3i-1})
\prod\limits_{\substack{1\leq i\leq 3k+1\text{ and }\\ i\not\equiv 4(\bmod\;6)}}(1-q^i)},
\end{equation}
where $\alpha_k:=\lfloor \frac{k+1}{2} \rfloor$. Note that, in the case $k=1$, $\alpha_1=1$ and \eqref{e6} reduces to \eqref{e3}.

For $k\leq 1$ we consider the sequence $\mathbf a[k]=(a[k]_1,a[k]_2,\ldots,a[k]_{3k+1})$, where 
\begin{equation}
a[k]_j=\begin{cases} j,& j\not\equiv 4(\bmod\;6) \\ \frac{j}{2}, & j\equiv 4(\bmod\;6) \end{cases}.
\end{equation}

\begin{prop}\label{p1}
Using the notations above, we have that:
$$\DD_k(n)= \sum\limits_{J\subset \{\alpha_k+1,\alpha_{k}+2,\cdots, k\}} p_{\mathbf a[k]}(n-m_J),$$
where $m_J=\sum_{i\in J}(3i-1)$ and $\alpha_k=\lfloor \frac{k+1}{2} \rfloor$.
\end{prop}

\begin{proof}
From \eqref{e6} it follows that
\begin{equation}\label{gogu}
\sum_{n=0}^{\infty} \DD_k(n)q^n = \sum_{J\subset \{\alpha_k,\alpha_{k+1},\cdots, k\}} 
\frac{q^{m_J}}{\prod\limits_{i=1}^{\lfloor \frac{k+1}{2} \rfloor}(1-q^{3i-1})
\prod\limits_{\substack{1\leq i\leq 3k+1\text{ and }\\ i\not\equiv 4(\bmod\;6)}}(1-q^i)}.
\end{equation}
On the other hand, from \eqref{gen} we deduce that
\begin{equation}\label{geny}
\sum_{n=0}^{\infty} p_{\mathbf a[k]}(n-m)q^n= \frac{q^m}{\prod\limits_{i=1}^{\lfloor \frac{k+1}{2} \rfloor}(1-q^{3i-1})
\prod\limits_{\substack{1\leq i\leq 3k+1\text{ and }\\ i\not\equiv 4(\bmod\;6)}}(1-q^i)}.
\end{equation}
The conclusion follows from \eqref{gogu} and \eqref{geny}.
\end{proof}

Let $D[k]=\lcm(\mathcal A_k)$, where $\mathcal A_k=\{1\leq j\leq 3k+1\;:\;j\not\equiv 4(\bmod\;6)\}$. For instance, 
$D[1]=\lcm\{1,2,3\}=6$, $D[2]=\lcm\{1,2,3,5,7\}=210$ etc. Note that 
\begin{equation}\label{ak}
\mathcal A_k=\{a[k]_1,a[k]_2,\ldots,a[k]_{3k+1}\}
\end{equation}
and thus $D[k]=\lcm\{a[k]_1,a[k]_2,\ldots,a[k]_{3k+1}\}$.

\begin{prop}\label{p2}
$\DD_k(n)$ is a quasi-polynomial of degree $3k$, with the period $D[k]$, i.e.
$$\DD_k(n)=f_{k,3k}(n)n^{3k}+\cdots+f_{k,1}(n)n+f_{k,0}(n),\text{ for all }n\geq n_0(k),$$
where $n_0(k)=\frac{(k-\alpha_k)(3k+3\alpha_k+1)}{2}$ and $f_{k,j}(n+D[k])=f_{k,j}(n)$ for all $n\geq n_0(k)$.

Moreover, we have that: 
$$f_{k,j}(n)=\sum\limits_{J\subset \{\alpha_k+1,\alpha_{k}+2,\cdots, k\}} d_{\mathbf a[k],j}(n-m_J).$$
\end{prop}

\begin{proof}
Note that $$
n_0(k)=m_{\{\alpha_k+1,\alpha_k+2,\ldots,k\}}=\sum_{i=\alpha_k+1}^k(3i-1)=3\left(\sum_{i=\alpha_k+1}^k i\right) - (k-\alpha_k)= $$
\begin{equation}\label{qqq}
= \frac{(k-\alpha_k)(3k+3\alpha_k+1)}{2}.
\end{equation}
The expression of $f_{k,j}(n)$ follows from Proposition \ref{p1} and \eqref{quasi}.

Now, the conclusion follows from \eqref{qqq} and the fact that  $d_{\mathbf a[k],j}(n+D[k])=d_{\mathbf a[k],m}(n)$ 
for all $0\leq j\leq 3k$ and $n\geq 0$.
\end{proof}

\begin{cor}\label{t1}
With the above notations, for $n\geq n_0(k)$ we have that
$$ f_{k,j}(n)=\frac{1}{(3k)!} \sum\limits_{J\subset \{\alpha_k+1,\alpha_{k}+2,\cdots, k\}} 
\sum_{\substack{ 0\leq j_i \leq \frac{D[k]_i}{a[k]_i}-1,\;1\leq i\leq 3k+1 \\ 
                 a[k]_1j_1+\cdots+a[k]_{3k+1}j_{3k+1} \equiv n-m_J (\bmod\;D[k])}} \sum_{\ell=j}^{3k} \stir{3k+1}{\ell+1} \times  $$
$$ \times (-1)^{\ell-j} \binom{\ell}{j}D[k]^{-\ell}(a[k]_1j_1+\cdots+a[k]_{3k+1}j_{3k+1})^{\ell-j}.$$
In particular, it follows that
$$ \DD_k(n) = \frac{1}{(3k)!} \sum_{\ell=0}^{3k} \sum\limits_{J\subset \{\alpha_k+1,\alpha_{k}+2,\cdots, k\}} 
\sum_{\substack{ 0\leq j_i \leq \frac{D[k]_i}{a[k]_i}-1,\;1\leq i\leq 3k+1 \\ 
                 a[k]_1j_1+\cdots+a[k]_{3k+1}j_{3k+1} \equiv n-m_J (\bmod\;D[k])}} \sum_{\ell=j}^{3k} \stir{3k+1}{\ell+1} \times  $$
$$ \times (-1)^{\ell-j} \binom{\ell}{j}D[k]^{-\ell}(a[k]_1j_1+\cdots+a[k]_{3k+1}j_{3k+1})^{\ell-j} (n-m_J)^j.$$
\end{cor}

\begin{proof}
The conclusion follows from Proposition \ref{p2} and Theorem \ref{dam}.
\end{proof}

\section{New formulas for the number of plane partition diamonds of length $k$ of $n$.}

We recall that $\alpha_k=\lfloor \frac{k+1}{2} \rfloor$. We define
$$\beta_k:=\begin{cases} 5\alpha_k-2,&k\text{ is odd}\\ 5\alpha_k+1,&k\text{ is even} \end{cases}\text{ and }
\mathcal B_k=\{1,2,\ldots,\beta_k\}.$$ 
For instance, $\mathcal B_1=\{1,2,3\}$, $\mathcal B_2=\{1,2,3,4,5,6\}$ etc.

\begin{lema}\label{phi}
With the above notations, the map $$\varphi_k:\mathcal B_k\to \mathcal A_k,\; \varphi_k(j)=j+\left\lceil \frac{j-3}{5} \right\rceil,$$
is bijective. Moreover, the inverse of $\varphi_k$ is the map
$$\varphi_k^{-1}:\mathcal A_k\to\mathcal B_k,\;\varphi_k^{-1}(j)=j-\left\lceil \frac{j-3}{6} \right\rceil.$$
\end{lema}

\begin{proof}
Let $j\in \mathcal B_k$ and write $j=5i+r$ for $1\leq r\leq 5$ and $i\geq 0$. We have that:
\begin{align*}
& \varphi_k(5i+1)=6i+1,\;\varphi_k(5i+2)=6i+2,\;\varphi_k(5i+3)=6i+3,\\
& \varphi_k(5i+4)=6i+5,\;\varphi_k(5i+5)=6i+6.
\end{align*}
If $k=2p$, then $\alpha_k=p$ and $$\varphi_k(5\alpha_k+1)=\varphi_k(5p+1)=6p+1=3k+1 = \max\mathcal A_k.$$
If $k=2p+1$, then $\alpha_k=p+1$ and $$\varphi_k(5\alpha_k-2)=\varphi_k(5p+3)=6p+3=3k=\max\mathcal A_k.$$
Also $\varphi_k(1)=1$ and $\varphi_k$ is increasing, hence injective.

From the above considerations, it follows that $\varphi_k$ is bijective. 

Let $\psi_k:\mathcal B_k\to\mathcal A_k$, $\psi_k(j)=j-\left\lceil \frac{j-3}{6} \right\rceil$.
Let $j\in\mathcal A_k$. Then we can write $j=6i+r$, where $1\leq r\leq 6$ and $r\neq 4$.
We have that
\begin{align*}
& \psi_k(6i+1)=5i+1,\;\psi_k(6i+2)=5i+2,\;\psi_k(6i+3)=5i+3,\\
& \psi_k(6i+5)=5i+4,\;\psi_k(6i+5)=5i+5.
\end{align*}
Since $\psi_k(1)=1$ and $\psi_k(\max \mathcal A_k)=\max\mathcal B_k$, from the above identities, it
follows that $\psi_k$ is surjective and incresing. Hence, $\psi_k$ is bijective. 

The function $\psi_k\circ \varphi_k:\mathcal A_k\to\mathcal A_k$ is bijective and increasing, hence $\psi_k\circ\varphi_k$ is the
identity function of $\mathcal A_k$. Similarly, $\varphi_k\circ\psi_k$ is the identity function of $\mathcal B_k$.
Thus, $\psi_k=\varphi_k^{-1}$, as required.
\end{proof}

We consider the subset 
$$\mathcal B'_k=\{j\in\mathcal B_k\;:\; j \equiv 2,5 (\bmod\;5) \text{ and }j\leq \varphi_k^{-1}(3\alpha_k-1) \}.$$
We also let 
$$\varepsilon_k:\mathcal B_k \to \{1,2\},\; \varepsilon_k(j)=\chi_{\mathcal B'_k}(j)+1,$$ 
where $\chi_{\mathcal B'_k}$ is
the characteristic function of the subset $\mathcal B'_k$ of $\mathcal B_k$.
We also let 
$$s_k(t_1,t_2,\ldots,t_{\beta_k}):=\prod_{j\in \mathcal B'_k}\left(1+
\max\left\{t_j, 2\left(\frac{D[k]}{\varphi_k(j)}-1\right)-t_j \right\}\right).$$

\begin{teor}\label{main}
We have that:
\begin{align*}
& \DD_k(n)=\frac{1}{(3k)!} \sum_{\substack{J\subset \{\alpha_k,\alpha_{k+1},\cdots, k\}\text{ and }\\ 
(t_1,t_2,\ldots,t_{\beta_k})\in \mathbf{A}_k(m_J)}}
s_k(t_1,t_2,\ldots,t_{\beta_k})
\prod_{\ell=1}^{3k} \left(\frac{n- \sum\limits_{j=1}^{\beta_k}t_j\varphi_k(j) -m_J}{D[k]}+\ell \right),
\end{align*}
where $\mathbf{A}_k(m)=\{(t_1,t_2,\ldots,t_{\beta_k})\;:\;0\leq t_j \leq \varepsilon_k(j)\left(\frac{D[k]}{\varphi_k(j)}-1\right)
\text{ for all }1\leq j\leq\beta_k$ and  $\sum_{j=1}^{\beta_k}t_j\varphi_k(j) \equiv n-m(\bmod\;D[k])\}$.
\end{teor}

\begin{proof}
First, note that for $n<m_J$ we have that:
$$\prod_{\ell=1}^{3k} \left(\frac{n-a[k]_1j_{1}- \cdots -a[k]_{3k+1}j_{3k+1}-m_J}{D[k]}+\ell \right)=0.$$
Therefore, from Proposition \ref{p1} and Theorem \ref{pan} it follows that:
\begin{align*}
& \DD_k(n)=\frac{1}{(3k)!} \sum_{J\subset \{\alpha_k,\alpha_{k+1},\cdots, k\}} \sum_{\substack{0\leq j_1\leq \frac{D[k]}{a[k]_1}-1,\ldots, 0\leq j_{3k+1}\leq \frac{D[k]}{a[k]_{3k+1}}-1 \\ 
a[k]_1j_1+\cdots+a[k]_{3k+1}j_{3k+1} \equiv n-m_J (\bmod D[k])}} \\
&  \prod_{\ell=1}^{3k} \left(\frac{n-a[k]_1j_{1}- \cdots -a[k]_{3k+1}j_{3k+1}-m_J}{D[k]}+\ell \right).
\end{align*}
The conclusion follows from Lemma \ref{phi}.
\end{proof}

\begin{exm}(MacMahon's example)\rm\label{macex}
We consider 
$$\DD_1(n)=\#\{(a_1,a_2,a_3,a_4)\;:\;a_1+a_2+a_3+a_4=n,\;a_1\geq a_2,\;a_1\geq a_3,\;a_2\geq a_4\text{ and }a_3\geq a_4\}.$$
Comparing \eqref{e3} with \eqref{gen}, it follows that $\DD_1(n)=p_{(1,2,2,3)}(n)$ for all $n\geq 0$. Since $\alpha_1=1$ and $D[1]=6$, from Theorem \ref{main}, it follows that:
$$\DD_1(n)=\frac{1}{6} \sum_{\substack{0\leq t_1 \leq 5,\;0\leq t_2\leq 4,\;0\leq t_3\leq 1 \\ t_1+2t_2+3t_3 \equiv n(\bmod\;6)}} 
(\min\{t_2,4-t_2\}+1) \prod_{\ell=1}^6 \left( \frac{n-t_1-2t_2-3t_3}{6}+\ell \right).$$
\end{exm}

\section{The polynomial part and Sylvester waves of $\DD_k(n)$}

From \eqref{wave} and Proposition \ref{p1}, we can write
\begin{equation}\label{undi}
\DD_k(n)=\sum_{j=1}^{\infty} W_j(k,n),\text{ where }W_j(k,n)= \sum\limits_{J\subset \{\alpha_k+1,\alpha_{k}+2,\cdots, k\}} 
W_j(n-m_J,\mathbf a[k]),
\end{equation}
$m_J=\sum_{i\in J}(3i-1)$ and $\alpha_k=\lfloor \frac{k+1}{2} \rfloor$.
In particular, the \emph{polynomial part} of $\DD_{k}(n)$ is the function
\begin{equation}\label{poli}
\PP_{r}(n):= \sum\limits_{J\subset \{\alpha_k+1,\alpha_{k}+2,\cdots, k\}} P_{\mathbf a[k]}(n-m_J),
\end{equation}
where $P_{\mathbf a[k]}(n-m_J)=W_1(n-m_J,\mathbf a[k])$.

\begin{teor}\label{tunde}
With the above notations we have that
$$W_j(k,n)=\frac{1}{D[k](3k)!} \sum\limits_{J\subset \{\alpha_k+1,\alpha_{k}+2,\cdots, k\}} 
           \sum_{m=1}^{3k+1} \sum_{\ell=1}^{j} \rho_j^{\ell} \sum_{t=m-1}^{3k} 
\stir{3k+1}{t+1} (-1)^{t-m+1} \binom{t}{m-1} \times$$ 
$$\times \sum_{\substack{0\leq j_1\leq \frac{D[k]_1}{a[k]_1}-1,\ldots, 0\leq j_{3k+1}\leq \frac{D[k]_{3k+1}}{a[k]_{3k+1}}-1 \\ 
a[k]_1j_1+\cdots+a[k]_{3k+1}j_{3k+1} \equiv \ell (\bmod j)}} D^{-t} (a[k]_1j_1+\cdots+a[k]_{3k+1}j_{3k+1})^{t-m+1} (n-m_J)^{m-1}.
$$
\end{teor}

\begin{proof}
The conclusion follows from Proposition \ref{p1}, Proposition \ref{unde} and \eqref{undi}.
\end{proof}

\begin{teor}\label{t2}
With the above notations, we have that \small
\begin{align*}
& \PP_k(n)=\frac{1}{D[k](3k)!} \sum_{\substack{J\subset \{\alpha_k,\alpha_{k+1},\cdots, k\}\text{ and }\\ 
(t_1,t_2,\ldots,t_{\beta_k})\in \mathbf{B}_k(m_J)}}
s_k(t_1,t_2,\ldots,t_{\beta_k})
\prod_{\ell=1}^{3k} \left(\frac{n- \sum\limits_{j=1}^{\beta_k}t_j\varphi_k(j) -m_J}{D[k]}+\ell \right),
\end{align*} \normalsize
where $\mathbf{B}_k(m)=\{(t_1,t_2,\ldots,t_{\beta_k})\;:\;0\leq t_j \leq \varepsilon_k(j)\left(\frac{D[k]}{\varphi_k(j)}-1\right)
\text{ for all }1\leq j\leq\beta_k\}$.
\end{teor}

\begin{proof}
The proof is similar to the proof of Theorem \ref{main}, using 
Proposition \ref{p1}, Theorem \ref{Pan} and \eqref{poli}.
\end{proof}

\begin{teor}\label{t3}
With the above notations, we have that:
$$\PP_{k}(n):=\frac{1}{a[k]_1\cdots a[k]_{3k+1}} \sum\limits_{J\subset \{\alpha_k+1,\alpha_{k}+2,\cdots, k\}} 
              \sum_{u=0}^{3k}\frac{(-1)^u}{(3k-u)!} \times $$  $$ \times \sum_{i_1+\cdots+i_{3k+1}=u} 
\frac{B_{i_1}\cdots B_{i_{3k+1}}}{i_1!\cdots i_{3k+1}!}a[k]_1^{i_1}\cdots a[k]_{3k+1}^{i_{3k+1}} (n-m_J)^{3k-u}.$$
\end{teor}

\begin{proof}
The conclusion follows from Proposition \ref{p1}, Theorem \ref{Pan2} and \eqref{poli}.
\end{proof}

\begin{exm}(MacMahon's example revised)\rm\label{macexx}\rm
We consider $\DD_1(n)$; see Example \ref{macex}. From Theorem \ref{t2}, the polynomial part of $\DD_1(n)$ is 
$$\PP_1(n)=\frac{1}{36} \sum_{\substack{0\leq t_1 \leq 5,\;0\leq t_2\leq 4,\;0\leq t_3\leq 1}} 
(\min\{t_2,4-t_2\}+1) \prod_{\ell=1}^6 \left( \frac{n-t_1-2t_2-3t_3}{6}+\ell \right).$$
\end{exm}


\end{document}